\documentclass[12pt]{article}

\usepackage{a4wide}
\usepackage{amssymb,amsmath,array}

\allowdisplaybreaks[1]

\newcommand{\titre}{Automorphisms of quantum matrices}

\newenvironment{proof}{\begin{trivlist}\item[]{\it
Proof.}}{\hfill$\square$\end{trivlist}}

\newtheorem{theorem}{Theorem}[section]
\newtheorem{corollary}[theorem]{Corollary}

\newtheorem{lemma}[theorem]{Lemma}
\newtheorem{proposition}[theorem]{Proposition}

\newtheorem{conjecture}[theorem]{Conjecture}

\newcommand{\gc}{ [ \hspace{-0.65mm} [}
\newcommand{\dc}{]  \hspace{-0.65mm} ]}

\newcommand{\ia}{i,\alpha}

\newcommand{\aut}{{\rm Aut}}

\newcommand{\ideal}[1]{\langle {#1}\rangle}

\def\N{{\mathbb{N}}}

\def\co{{\mathcal O}}

\def\oqmn{\co_q(M_n)}

\def\oqmtwo{\co_q(M_2)}
\def\oqmm13{\co_q(M_{1,3})}
\def\oqm23{\co_q(M_{2,3})}
\def\oqmmn{\co_q(M_{m,n})}

\def\ia{i,\alpha}

\input{xy}
\xyoption{all}

\begin{document}

\title{\titre} \author{S Launois\thanks{The first author is grateful for the financial
support of EPSRC first grant \textit{EP/I018549/1}.}~ and T H Lenagan}
\date{}

\maketitle

\begin{center}
{\small {\it This paper is dedicated to Kenny Brown and Toby Stafford for their $|A_5|$-th birthday.}}
\end{center}

\begin{abstract}
We study the automorphism group of the algebra $\oqmn$ of $n \times n$ generic quantum matrices. We provide evidence for our conjecture that this group is generated by the transposition and the subgroup of those automorphisms acting on the canonical generators of $\oqmn$ by multiplication by scalars. Moreover, we prove this conjecture in the case when $n=3$.
\end{abstract}




\section*{Introduction}

Let $\mathbb{K}$ be a field and $q$ be an element in 
$\mathbb{K}^*:= \mathbb{K}\setminus \{0\}$. We assume that $q$ is not a root of unity. 
The quantisation of the ring of
regular functions on $m \times n$ matrices with entries in 
$ \mathbb{K}$ is denoted by $\oqmmn$; this  is
the $ \mathbb{K}$-algebra generated by the $m \times n $ indeterminates
$Y_{\ia}$, $1 \leq i \leq m$ and $ 1 \leq \alpha \leq n$, subject to the
following relations:\\ \[
\begin{array}{ll}
Y_{i, \beta}Y_{i, \alpha}=q^{-1} Y_{i, \alpha}Y_{i, \beta},
& (\alpha < \beta); \\
Y_{j, \alpha}Y_{i, \alpha}=q^{-1}Y_{i, \alpha}Y_{j, \alpha},
& (i<j); \\
Y_{j,\beta}Y_{i, \alpha}=Y_{i, \alpha}Y_{j,\beta},
& (i <j,  \alpha > \beta); \\
Y_{j,\beta}Y_{i, \alpha}=Y_{i, \alpha} Y_{j,\beta}-(q-q^{-1})Y_{i,\beta}Y_{j,\alpha},
& (i<j,  \alpha <\beta). 
\end{array}
\]
It is well known that $\oqmmn$ is a Noetherian domain that can be 
presented as an iterated Ore extension over the base field $\mathbb{K}$ 
with the indeterminates $Y_{\ia}$ adjoined in lexicographic order. 
Moreover, as all the defining relations of the algebra are quadratic, 
$\oqmmn$ is a graded 
algebra with all the indeterminates $Y_{\ia}$ in degree 1.

This article is concerned with the symmetries of quantum matrices. More
precisely, we are studying the automorphism group of this family of algebras.
As usual in the quantum setting, it is to be 
expected that the automorphism group of
$\oqmmn$ is quite small (see for instance \cite{ll} and references therein).

In the case of $\oqmmn$, there are two classes of automorphisms that are well known:
\begin{enumerate}
\item The set $\mathcal{H}$ of automorphisms acting on the indeterminates $Y_{\ia}$ by multiplication by nonzero scalars; this subgroup of $\aut(\oqmmn)$ is isomorphic to the torus $(\mathbb{K}^*)^{m+n-1}$ \cite[Corollary 4.11 and its proof]{ll};
\item In the square case, where $m=n$, the transposition $\tau$ sending $Y_{\ia}$ to $Y_{\alpha, i}$ is an automorphism that generates a subgroup of order 2 of $\aut(\oqmn)$.
\end{enumerate}

In the case where $m \neq n$, we proved in \cite{ll} that
$\aut(\oqmmn)=\mathcal{H}$. Unfortunately, the methods used in that article 
are not sufficient to resolve the square case. 
However, it was proved by Alev and
Chamarie \cite{ac} that $\aut(\oqmtwo) = \mathcal{H} \rtimes \langle \tau
\rangle$. In view of these results, it is natural to conjecture the following
result. 

\begin{conjecture}
\label{conj}
$\aut(\oqmn) = \mathcal{H} \rtimes \langle \tau \rangle.$
\end{conjecture}

The main aim of this article is to provide evidence for this conjecture, and
also to prove it in the case when $n=3$. 

Set $R:=\oqmn$,  $G:=\mathcal{H} \rtimes \langle \tau \rangle$, and 
let $\sigma \in \aut(R)$. In Section \ref{reductionstep}, we prove that 
there exists $g \in G$ such that:
\begin{equation}
\label{eqIntro}
g\circ \sigma (Y_{\ia}) -Y_{\ia} \mbox{ is a sum of homogeneous terms of 
degree }\geq 2.
\end{equation}
Of course, we conjecture that $g\circ \sigma =\mathrm{id}$. 
The above result (\ref{eqIntro}) already has interesting consequences. 
Indeed, it follows from a result of Alev and Chamarie 
\cite[Lemme 1.4.2]{ac} that such a 
$g \circ \sigma$ belongs to the subalgebra 
of $\mathrm{End}_{\mathbb{K}}(R)$ generated by the derivations of $R$. 
As the derivations of $R$ were computed in \cite{ll2}, we can for instance 
prove that every normal element of $R$ is fixed by $g \circ \sigma$ 
(an element $u$ is  {\em normal} in $R$ if $uR=Ru$). 

Before going any further, let us mention that the normal elements of $R$ have been described in \cite{ll}. They are closely related to distinguished elements of $R$ called {\it quantum minors}. Recall that if $I:=\{i_1<\dots <i_t \},\Lambda=\{\alpha_1 < \dots < \alpha_t\} \subseteq \{1, \dots ,n \}$ with $|I|=|\Lambda|=t \neq 0$, then the quantum minor $[I|\Lambda]=[i_1, \dots ,i_t|\alpha_1, \dots ,\alpha_t]$ is defined by:
$$[I|\Lambda]=[i_1, \dots ,i_t|\alpha_1, \dots ,\alpha_t]:=\sum_{w \in S_t} (-q)^{l(w)} Y_{i_1, \alpha_{w(1)}}Y_{i_2, \alpha_{w(2)}}\cdots Y_{i_t ,\alpha_{w(t)}},$$
where $l$ is the usual length function on permutations.

It is well known that the quantum minors $b_i$ with $i \in \{1, \dots ,
2n-1\}$ defined by 
$$
b_i:=\left\{ \begin{array}{ll} [1,\dots ,i | n-i+1, \dots
,n] & \mbox{ if }1\leq i \leq n \\
~ [i-n+1, \dots ,n | 1, \dots ,2n-i ] & \mbox{ otherwise,} ~\\ 
\end{array} \right.
$$
are normal in $R$, so that the main result of Section \ref{reductionstep}
shows that 
$$
g \circ \sigma (b_i)=b_i, ~\mbox{for all }i \in \{1, \dots ,
2n-1\}.
$$ 
Note that $\Delta:=b_n$ is the so-called {\it quantum determinant} of $R$. As
we assume that $q$ is not a root of unity, the centre of $R$ is precisely the
polynomial algebra in the quantum determinant $\Delta$, and so the
previous result shows in particular that every element in the centre of $R$ is
left invariant by $g\circ \sigma$.

In Section 2, we use (\ref{eqIntro}) as well as graded arguments in order to
prove that when $n=3$ we indeed have $g\circ \sigma=\mathrm{id}$, so that
Conjecture \ref{conj} is true in this case.

Throughout this paper, we set $\gc a,b \dc := \{ i\in{\mathbb N} \mid a\leq
i\leq b\}$ and we assume $n \geq 3$. 



\section{The automorphism group of $\oqmn$: Reduction step}
\label{reductionstep}

In this section, we investigate the group of automorphisms of $R=\oqmn$. We
will be using graded arguments, as well as the 
induced actions of $\aut(R)$ on
the set of height one prime ideals, on the centre and on the set of
normal elements of $R$.


In the sequel, we will use several times the following well-known result
concerning normal elements of $R=\oqmn$.

\begin{lemma}
\label{utile}
Let $u$ and $v$ two nonzero normal elements of $R$ such that $\ideal{u}=\ideal{v}$. Then
there exist $\lambda,\mu \in  \mathbb{K}^*$ such that $u=\lambda v $ and $v =\mu u$.
\end{lemma}


\subsection{Torus automorphisms of  $\oqmn$}

Recall from the Introduction that $\mathcal{H}$ denote the subgroup of those
automorphisms of $R$ acting on the indeterminates $Y_{\ia}$ by multiplication
by nonzero scalars. The proof of \cite[Corollary 4.11]{ll} shows that
$\mathcal{H}$ is isomorphic to the torus $(\mathbb{K}^*)^{2n-1}$. More
precisely, for any $h:=(a_1,\dots,a_n,b_1,\dots,b_{n-1})\in
(\mathbb{K}^*)^{2n-1}$, define an automorphism $\sigma_h$ in $\mathcal{H}$ as
follows: $$\sigma_{h}( Y_{\ia})=\left\{ \begin{array}{ll} a_ib_{\alpha}Y_{\ia}
& \mbox{ if } \alpha <n \\ a_iY_{\ia} & \mbox{ if } \alpha =n. \end{array}
\right. $$ The proof of \cite[Corollary 4.11]{ll} shows that the map $h
\mapsto \sigma_h$ from $(\mathbb{K}^*)^{2n-1}$ to $\mathcal{H}$ is an
isomorphism. The elements of $\mathcal{H}$, that is the automorphisms
$\sigma_h$ with $h \in (\mathbb{K}^*)^{2n-1}$, are called the {\it torus
automorphisms} to $R$.


\subsection{Height one prime ideals of $\oqmn$}

In \cite[Propositions 3.5 and 3.6]{ll}, we have described the height one
primes of $R$. We now recall the results that we have obtained. 

\begin{proposition} \label{0stratum}
For any height one prime ideal $P$ of $\oqmn$, there exists an irreducible polynomial $V=\sum_{i_1=0}^{r_1} \dots \sum_{i_n=0}^{r_n}
a_{i_1,\dots,i_n} X_1^{i_1} \dots X_n^{i_n} \in  \mathbb{K}[X_1,\dots,X_n]$ (where $r_i = \deg_{X_i}V$ for all $i \in \{1,\dots,n\}$) such that
$P = \ideal{u} $, where 
$$
u:=\sum_{i_1=0}^{r_1} \dots
\sum_{i_n=0}^{r_n} a_{i_1,\dots,i_n} \prod_{j=1}^n b_{j}^{i_j}
b_{n+j}^{r_j-i_j}.  
$$ 
(By convention, we set $b_{2n}:=1$.)\\
Moreover, $u$ is normal in $R$. 
\end{proposition}



\subsection{$q$-commutation, gradings and automorphisms}

Recall that the relations that define $R=\oqmn$ are all quadratic, so that
$R=\oplus_{i \in \N} R_i$ is a $\mathbb{N}$-graded algebra, the canonical
generators $Y_{i,\alpha}$ of $R$ having degree one.  Note, for later use, that a
$t \times t$ quantum minor of $R$ is a homogeneous element of degree $t$
with respect to this grading of $R$. In the sequel, $R$ will always be
endowed with this grading.

In \cite[Corollary 4.3]{ll}, we have shown the following result.

\begin{proposition}
 \label{graduation2} Let $\sigma$ be an automorphism of $R=\oqmn$ and $x$ an
 homogeneous element of degree $d$ of $R$. Then $\sigma(x)=y_d + y_{> d}$,
 where $y_d \in R_d \setminus \{0\}$ and $y_{>d} \in R_{> d}$. 
 \end{proposition}

Note that the torus automorphisms of $R$ preserve degrees. We finish this section by recording the following result for later use.

\begin{lemma}
\label{toursreduction}
Let $\sigma \in \aut(R)$ such that there exist nonzero scalars $\lambda_{\ia}$ with 
$$\sigma(Y_{\ia})-\lambda_{\ia}Y_{\ia} \in R_{\geq 2} \mbox{ for all }(\ia). $$
Then there exists a torus automorphism $\sigma_h \in \mathcal{H}$ such that 
$$\sigma_h \circ \sigma(Y_{\ia})-Y_{\ia} \in R_{\geq 2} \mbox{ for all }(\ia). $$  
\end{lemma}
\begin{proof} 
Assume $i<j$ and $ \alpha <\beta$. Applying $\sigma$ to the relation $Y_{j,\beta}Y_{i, \alpha}=Y_{i, \alpha} Y_{j,\beta}-(q-q^{-1})Y_{i,\beta}Y_{j,\alpha}$, and then identifying the degree 2 components, yields:
$$\lambda_{\ia}\lambda_{j,\beta}=\lambda_{i,\beta}\lambda_{j,\alpha}$$
for all $i<j$ and $ \alpha <\beta$. 
Hence, the matrix $(\lambda_{\ia})$ has rank one, so that there 
exist $a_1,\dots,a_n,b_1,\dots,b_{n-1},b_n=1 \in \mathbb{K}^*$ such that 
$$\lambda_{\ia}=a_ib_{\alpha}$$
for all $(\ia)$. 
Set $h=(a_1^{-1},\dots,a_n^{-1},b_1^{-1},\dots,b_{n-1}^{-1}) 
\in (\mathbb{K}^*)^{2n-1}$. 
Then one easily checks that the automorphism $\sigma_h \in \mathcal{H}$ 
has the property that
$\sigma_h \circ \sigma(Y_{\ia})-Y_{\ia} \in R_{\geq 2} 
\mbox{ for all }(\ia).$ 
\end{proof}

\subsection{Automorphism group of $\oqmn$: action on the centre}

Recall that the centre of $R=\oqmn$ is the polynomial ring
 $ \mathbb{K}[\Delta]$,
 where $\Delta$ denotes the quantum determinant of $R$. 
We now apply the results of the previous section to $R=\oqmn$ to prove that the quantum determinant 
$\Delta$ of $R$ is an eigenvector of every automorphism of $R$.

\begin{proposition}
\label{centreaction}
Let $\sigma$ be an automorphism of $R$. Then there exists $\mu \in  \mathbb{K}^*$ such that 
$\sigma(\Delta) =\mu \Delta$.
\end{proposition}
\begin{proof}
Since $\sigma$ is an automorphism of $R$, it induces an automorphism of the centre $ \mathbb{K}[\Delta]$ of $R$. 
Hence there exist $\mu \in  \mathbb{K}^*$ and $\lambda \in  \mathbb{K}$
such that $\sigma( \Delta ) = \mu \Delta + \lambda$.
 Moreover, $\Delta$ is an homogeneous element of degree $n$ of
 $R=\oqmn$. Hence, Proposition \ref{graduation2} shows that we must have $\sigma( \Delta
 ) \in R_{\geq n}$. Naturally, this forces $\lambda$ to be zero.
\end{proof}

\subsection{Automorphism group of $\oqmn$: action on the normal
  element $b_1=Y_{1,n}$}

\begin{lemma}
\label{actiony1n}
Let $\sigma \in \aut(R)$. Then there exist $\epsilon \in \{0,1\}$,
$P,Q,P',Q' \in  \mathbb{K}[X]$ such that 
$$\tau^{\epsilon} \circ \sigma (Y_{1,n})=P(\Delta)b_{1} + Q(\Delta) b_{n+1} \mbox{ \ and \  } \sigma^{-1}\circ \tau^{\epsilon} (Y_{1,n})=P'(\Delta)b_{1} + Q'(\Delta) b_{n+1}
.$$
\end{lemma}
\begin{proof} As $\ideal{b_1}=\ideal{Y_{1,n}}$ is a height one
  prime ideal of $R$, the ideal 
  $\ideal{\sigma(b_1)}$ must also be a height one
  prime of $R$. It follows from Proposition \ref{0stratum} that 
$\ideal{\sigma(Y_{1,n})} = \ideal{u} $
where 
$$
u:=\sum_{i_1=0}^{r_1} \dots
\sum_{i_n=0}^{r_n} a_{i_1,\dots,i_n} \prod_{j=1}^n b_{j}^{i_j}
b_{n+j}^{r_j-i_j}  
$$ 
is normal in $R$. Hence, we deduce from Lemma \ref{utile} that 
$$\sigma(Y_{1,n})= \lambda u = \sum_{i_1=0}^{r_1} \dots
\sum_{i_n=0}^{r_n} a'_{i_1,\dots,i_n} \prod_{j=1}^n b_{j}^{i_j}
b_{n+j}^{r_j-i_j},$$
where $\lambda \in  \mathbb{K}^*$ and $a'_{i_1,\dots,i_n}:=\lambda
a_{i_1,\dots,i_n}$. 

On the other hand, it follows from Proposition \ref{graduation2} that 
$\sigma(Y_{1,n})= u_1 +u_{\geq 2}$, with $u_1 \in R_1\setminus\{0\}$ and $u_{\geq 2}
\in R_{\geq 2}$. Since $b_i$ is homogeneous of degree $i$ if $i \leq
n$ , and $2n-i$ if $i \geq n$, comparing the two expresssions of 
$\sigma(Y_{1,n})$ that we have obtained leads to:\\

Either $\sigma (Y_{1,n})=P(\Delta)b_{1} + Q(\Delta) b_{n+1}$ 
or $\sigma (Y_{1,n})=P(\Delta)b_{n-1} + Q(\Delta) b_{2n-1}$. \\

\noindent Now, the existence of $\epsilon$ such that $\tau^{\epsilon} \circ
\sigma (Y_{1,n})=P(\Delta)b_{1} + Q(\Delta) b_{n+1}$ easily follows from the
fact that $\tau(\Delta)=\Delta$, and $\tau(b_i)=b_{2n-i}$ for all $i$. 

Note that the previous reasoning applies also to $\sigma^{-1}\circ
\tau^{\epsilon}$, so that $\sigma^{-1}\circ \tau^{\epsilon}
(Y_{1,n})=P'(\Delta)b_{1} + Q'(\Delta) b_{n+1}$ or $\sigma^{-1}\circ
\tau^{\epsilon}(Y_{1,n})=P'(\Delta)b_{n-1} + Q'(\Delta) b_{2n-1}$. Recall, 
from
Proposition \ref{centreaction}, 
that there exists $\mu \in \mathbb{K}^*$ such
that $\sigma^{-1}\circ \tau^{\epsilon}(\Delta)=\mu \Delta$, so that applying
$\sigma^{-1}\circ \tau^{\epsilon}$ to $\tau^{\epsilon} \circ \sigma
(Y_{1,n})=P(\Delta)b_{1} + Q(\Delta) b_{n+1}$ leads to $$Y_{1,n}=P(\mu \Delta)
\sigma^{-1}\circ \tau^{\epsilon} (Y_{1,n}) +Q(\mu \Delta)\sigma^{-1}\circ
\tau^{\epsilon}(b_{n+1}).$$ Comparing the degree one part of each side using
Proposition \ref{graduation2}, this easily implies that the case
$\sigma^{-1}\circ \tau^{\epsilon}(Y_{1,n})=P'(\Delta)b_{n-1} + Q'(\Delta)
b_{2n-1}$ is impossible, so that $\sigma^{-1}\circ \tau^{\epsilon}
(Y_{1,n})=P'(\Delta)b_{1} + Q'(\Delta) b_{n+1}$, as desired.
\end{proof}

\subsection{Automorphism group of $\oqmn$: reduction step.}

In view of Lemma \ref{actiony1n}, it is natural to introduce $$G':=\{ \sigma
\in \aut(R) \mid \sigma(Y_{1,n})=P(\Delta)b_{1} + Q(\Delta) b_{n+1}\}.$$ Note
that the proof of the previous lemma shows that $G'$ is invariant under taking inverses.

\begin{lemma}
\label{actionJ}
Set $J_r:=Y_{1,1}R+Y_{1,2}R+\dots+Y_{1,n-1}R$ and 
$J_c:=Y_{2,n}R+Y_{3,n}R+\dots+Y_{n,n}R$. If $\sigma \in G'$, 
then $\sigma(J_r)=J_r$ and $\sigma(J_c)=J_c$.
\end{lemma}
\begin{proof} The proof is given for the case $J:=J_r$; the proof for $J_c$ 
is similar. 

Let $\beta \in \gc 1,n-1 \dc$ and write $\sigma(Y_{1,\beta})$ in the
PBW basis of $R$:
$$
\sigma(Y_{1,\beta})= \sum_{\underline{\gamma} \in \Gamma}
c_{\underline{\gamma}} Y_{1,1}^{\gamma_{1,1}} Y_{1,2}^{\gamma_{1,2}}
\dots Y_{n,n}^{\gamma_{n,n}},
$$
where $\Gamma$ is a finite subset of $\mathbb{N}^{n^2}$ and each 
$c_{\underline{\gamma}}\neq 0$.  
Recall that $Y_{1,n}Y_{1,\beta}=q^{-1}Y_{1,\beta}Y_{1,n}$. Hence,
applying $\sigma$ to this equality leads to 
\begin{eqnarray*}
\lefteqn{
\left( P(\Delta)Y_{1,n} + Q(\Delta)[2,\dots,n  \mid
  1,\dots,n-1]\right) \left( \sum_{\underline{\gamma} \in \Gamma}
c_{\underline{\gamma}} Y_{1,1}^{\gamma_{1,1}} Y_{1,2}^{\gamma_{1,2}}
\dots Y_{n,n}^{\gamma_{n,n}} \right)=}  \\
&&
q^{-1} \left( \sum_{\underline{\gamma} \in \Gamma}
c_{\underline{\gamma}} Y_{1,1}^{\gamma_{1,1}} Y_{1,2}^{\gamma_{1,2}}
\dots Y_{n,n}^{\gamma_{n,n}} \right) \left( P(\Delta)Y_{1,n} +
Q(\Delta)[2,\dots,n  \mid 1,\dots,n-1]\right).
\end{eqnarray*} 
Now, since $\Delta$ is central in $R$, and $[2,\dots,n  \mid
  1,\dots,n-1]Y_{1,n}^{-1}=b_{n+1}b_1^{-1}$ is central in the field of
fractions of $R$, see \cite[Theorem 3.4]{ll}, we obtain 
 $$
 Y_{1,n}\left( \sum_{\underline{\gamma} \in \Gamma}
c_{\underline{\gamma}} Y_{1,1}^{\gamma_{1,1}} Y_{1,2}^{\gamma_{1,2}}
\dots Y_{n,n}^{\gamma_{n,n}} \right)
=q^{-1} \left( \sum_{\underline{\gamma} \in \Gamma}
c_{\underline{\gamma}} Y_{1,1}^{\gamma_{1,1}} Y_{1,2}^{\gamma_{1,2}}
\dots Y_{n,n}^{\gamma_{n,n}} \right) Y_{1,n}; 
$$ 
that is, 
\begin{eqnarray*}
\lefteqn{
\left( \sum_{\underline{\gamma} \in \Gamma} q^{-\gamma_{1,1}-\dots
  -\gamma_{1,n-1}+\gamma_{2,n}+\dots + \gamma_{n,n}}
c_{\underline{\gamma}} Y_{1,1}^{\gamma_{1,1}} Y_{1,2}^{\gamma_{1,2}}
\dots Y_{n,n}^{\gamma_{n,n}} \right) Y_{1,n}=}\hspace*{30ex}\\
&&
q^{-1} \left( \sum_{\underline{\gamma} \in \Gamma}
c_{\underline{\gamma}} Y_{1,1}^{\gamma_{1,1}} Y_{1,2}^{\gamma_{1,2}}
\dots Y_{n,n}^{\gamma_{n,n}} \right) Y_{1,n}.
\end{eqnarray*} 
As $R$ is a domain, this implies that 
$$ \sum_{\underline{\gamma} \in \Gamma} q^{-\gamma_{1,1}-\dots
  -\gamma_{1,n-1}+\gamma_{2,n}+\dots + \gamma_{n,n}}
c_{\underline{\gamma}} Y_{1,1}^{\gamma_{1,1}} Y_{1,2}^{\gamma_{1,2}}
\dots Y_{n,n}^{\gamma_{n,n}} =q^{-1}  \sum_{\underline{\gamma} \in \Gamma}
c_{\underline{\gamma}} Y_{1,1}^{\gamma_{1,1}} Y_{1,2}^{\gamma_{1,2}}
\dots Y_{n,n}^{\gamma_{n,n}} . $$
Identifying these two expressions in the PBW basis, and then using the fact that $q$ is
not a root of unity leads to  
$$-\gamma_{1,1}-\dots
  -\gamma_{1,n-1}+\gamma_{2,n}+\dots + \gamma_{n,n}=-1$$
for all $\underline{\gamma} \in \Gamma$. In particular, for all
$\underline{\gamma} \in \Gamma$, there exists $\beta_0 \in \{ 1, \dots ,n-1
\}$ such that $\gamma_{1,\beta_0} \geq 1 $. Hence 
$\sigma(Y_{1,\beta}) $ belongs to $J$, and so $\sigma(J)\subseteq J$. 

One can also apply this argument to $\sigma^{-1}$, so that we also
have  $\sigma^{-1}(J) \subseteq J$. From these two inclusions, we
conclude that  $\sigma(J)= J$.
\end{proof}

\begin{corollary}
\label{actionK}
Set $K_r:=\ideal{Y_{1,1},Y_{1,2},...,Y_{1,n}}=Y_{1,1}R+Y_{1,2}R+\dots+Y_{1,n}R$ and $K_c:=\ideal{Y_{1,n},Y_{2,n},...,Y_{n,n}}=Y_{1,n}R+Y_{2,n}R+\dots+Y_{n,n}R$. If $\sigma \in G'$, then $\sigma(K_r)=K_r$ and $\sigma(K_c)=K_c$.
\end{corollary}
\begin{proof} Again, we only consider the case of $K=K_r$. 

As $J=J_r \subset K$, Lemma \ref{actionJ} shows that $J \subset \sigma(K)$.

On the other hand, $K$ is a height $n$ prime ideal of $R$, so that 
 $\sigma(K)$ is also a height $n$ prime ideal. Moreover, since  $J \subset
\sigma(K)$, $Y_{1,1},Y_{1,2},...,Y_{1,n-1}$ belong to $\sigma(K)$. 
Now, $(q-q^{-1})Y_{1,n}Y_{i,1}=Y_{1,1}Y_{i,n}-Y_{i,n}Y_{1,1}
 \in \sigma(K)$ for all $i \in \gc 2,n \dc$. As $\sigma(K)$ is
 (completely) prime, this leads to: 
 either $Y_{1,n} \in \sigma(K)$ or $Y_{i,1} \in \sigma(K)$, for all 
 $i \in \gc 2,n \dc$.

We claim that the second possibility cannot happen. If it did then 
$\sigma(K)$ would strictly contain the ideal generated by the 
  $Y_{i,1}$, for $i \in \gc 1,n \dc$. However, this ideal is prime and has
height $n$, the same height as $\sigma(K)$. This is impossible.

Hence, $Y_{1,n} \in \sigma(K)$. As we already know that 
 $Y_{1,1},Y_{1,2},...,Y_{1,n-1}$ belong to $\sigma(K)$, 
 we obtain that 
$K \subseteq \sigma(K)$. Now these two ideals are prime and each has 
height $n$, so that they are equal; that is, 
 $\sigma(K)=K.$
\end{proof}

\begin{proposition}
\label{reduction}
Let $G$ be the subgroup of $\aut(R)$ generated by $\tau$ and the torus automorphisms. 
Let $\sigma \in \aut(R)$. Then there exists $g \in G$ such that, for
all $(\ia) \in \gc 1,n \dc^2$, we have 
$$g \circ \sigma (Y_{\ia}) - Y_{\ia} \in R_{\geq 2}.$$
\end{proposition}

\begin{proof} In view of Lemma \ref{toursreduction}, 
it is enough to prove that 
there exist $g\in G$ and 
nonzero scalars $\lambda_{\ia}$ with 
$$
g\circ\sigma(Y_{\ia})-\lambda_{\ia}Y_{\ia} 
\in R_{\geq 2} \mbox{ for all }(\ia). 
$$

First,  it follows from Lemma \ref{actiony1n} that there exist  $g' \in G$,
and $P,Q \in  \mathbb{K}[X]$ such that 
$$g' \circ \sigma (Y_{1,n})=P(\Delta)b_{1} + Q(\Delta) b_{n+1}
=P(\Delta)Y_{1,n} + Q(\Delta)[2,\dots,n  \mid 1,\dots,n-1].$$
Hence, it is enough to prove Proposition \ref{reduction} when $\sigma$ is an
automorphism of $R$ such that 
$$\sigma (Y_{1,n}) =P(\Delta)Y_{1,n} + Q(\Delta)[2,\dots,n  \mid
  1,\dots,n-1];$$
that is, when $\sigma \in G'$. 

So, let $\sigma \in G'$. It follows from Corollary \ref{actionK} that 
$\sigma(K_r)=K_r$. Hence, $\sigma$ induces an automorphism of 
$R/K_r$ an algebra 
that is isomorphic to $ \co_q(M_{n-1,n})$ via 
an isomorphism that sends 
$Y_{\ia}+K_r$ to $y_{i-1,\alpha}$, where $y_{\ia}$ denote the canonical 
generators of $ \co_q(M_{n-1,n})$. Hence, it follows from \cite{ll} that 
there exist $\lambda_{\ia} \in  \mathbb{K}^*$ such that
$$
\sigma(Y_{\ia})- \lambda_{\ia} Y_{\ia} \in K_r,
$$  
for all $(\ia) \in \gc 2,n \dc \times \gc 1,n \dc$.

Let $(\ia) \in \gc 2,n \dc \times \gc 1,n \dc$. Then there exist 
$\mu_1,\dots,\mu_n \in  \mathbb{K}$ and $u_{\geq 2} \in R_{\geq 2}$ such that 
\begin{eqnarray}
\label{eq1}
\sigma(Y_{\ia})= \lambda_{\ia} Y_{\ia} +\mu_1 Y_{1,1}+\dots + \mu_n
Y_{1,n}+u_{\geq 2}.
\end{eqnarray}

Similarly, using the fact that $\sigma(K_c)=K_c$, we obtain that 
for all $(\ia) \in \gc 1,n \dc \times \gc 1,n-1 \dc$,  there exist $\lambda'_{\ia} \in  \mathbb{K}^*$, 
$\mu'_1,\dots,\mu'_n \in  \mathbb{K}$ and $u'_{\geq 2} \in R_{\geq 2}$ such that 
\begin{eqnarray}
\label{eq2}
\sigma(Y_{\ia})= \lambda'_{\ia} Y_{\ia} +\mu'_1 Y_{1,n}+\dots + \mu'_n
Y_{n,n}+u'_{\geq 2}.
\end{eqnarray}

Comparing Equations (\ref{eq1}) and (\ref{eq2}), we obtain that 
for all $(\ia) \in \gc 2,n \dc \times \gc 1,n-1 \dc$,  there exist $\lambda_{\ia} \in  \mathbb{K}^*$, $ \mu_{\ia} \in  \mathbb{K}$ and $v_{\geq 2} \in R_{\geq 2}$ such that 
\begin{eqnarray}
\label{eq3bis}
\sigma(Y_{\ia})= \lambda_{\ia} Y_{\ia} +\mu_{\ia} Y_{1,n} +v_{\geq 2}.
\end{eqnarray}
Now, assume that $(\ia) \in \gc 2,n \dc \times \gc 1,n-2 \dc$. Applying $\sigma$ to $Y_{\ia}Y_{i,\alpha+1}=qY_{i,\alpha+1}Y_{\ia}$, and identifying the degree 2 terms, leads to 
$$(\lambda_{\ia} Y_{\ia} +\mu_{\ia} Y_{1,n})(\lambda_{i,\alpha+1} Y_{i,\alpha+1} +\mu_{i,\alpha+1} Y_{1,n})=q(\lambda_{i,\alpha+1} Y_{i,\alpha+1} +\mu_{i,\alpha+1} Y_{1,n})(\lambda_{\ia} Y_{\ia} +\mu_{\ia} Y_{1,n})$$
thanks to (\ref{eq3bis}). Using the commutation relations in $R$, we get: 
$$ (1-q)\lambda_{\ia}\mu_{i,\alpha+1} Y_{\ia} Y_{1,n}+(1-q) \lambda_{i,\alpha+1} \mu_{\ia} Y_{i,\alpha+1} Y_{1,n} +(1-q)\mu_{\ia} \mu_{i,\alpha+1} Y_{1,n}^2=0. $$
As $q-1 \neq 0$ and $\lambda_{\ia}\lambda_{i,\alpha+1} \neq 0$, this forces $\mu_{\ia} =0$ and $\mu_{i,\alpha+1}=0$. Hence, we have just proved that for all $(\ia) \in \gc 2,n \dc \times \gc 1,n-1 \dc$,  there exist $\lambda_{\ia} \in  \mathbb{K}^*$, and $v_{\geq 2} \in R_{\geq 2}$ such that 
\begin{eqnarray*}
\label{eq3}
\sigma(Y_{\ia})= \lambda_{\ia} Y_{\ia}  +v_{\geq 2},
\end{eqnarray*}
as required. 

Now let $i \in \gc 2,n  \dc$. As $Y_{i,n}Y_{1,n}=q^{-1}Y_{1,n}Y_{i,n}$, 
we must have 
$$\sigma(Y_{i,n})\sigma(Y_{1,n})=q^{-1}\sigma(Y_{1,n})\sigma(Y_{i,n});$$
that is, 
\begin{eqnarray*}
\lefteqn{
\left( \lambda_{i,n} Y_{i,n} +\mu_1 Y_{1,1}+\dots + \mu_n
Y_{1,n}+u_{\geq 2} \right) \left( P(\Delta)b_{1} + Q(\Delta) b_{n+1}
\right)=}\hspace*{10ex}\\
&&
q^{-1} \left( P(\Delta)b_{1} + Q(\Delta) b_{n+1} \right) \left( \lambda_{i,n} Y_{i,n} +\mu_1 Y_{1,1}+\dots + \mu_n
Y_{1,n}+u_{\geq 2} \right).
\end{eqnarray*}

As $\Delta$ and $b_{n+1}b_1^{-1}$ are central in the field of
fractions of $R$, we obtain  
$$\left( \lambda_{i,n} Y_{i,n} +\mu_1 Y_{1,1}+\dots + \mu_n
Y_{1,n}+u_{\geq 2} \right) b_1
=q^{-1} b_{1} \left( \lambda_{i,n}Y_{i,n} +\mu_1 Y_{1,1}+\dots + \mu_n
Y_{1,n}+u_{\geq 2} \right).$$
One can 
easily check that this forces $\mu_1=\dots=\mu_{n}=0$. 

Hence, for all $i \in \gc 2,n \dc $,  there exist $\lambda_{i,n} \in  \mathbb{K}^*$  such that 
\begin{eqnarray*}
\label{eq4}
\sigma(Y_{i,n}) - \lambda_{i,n} Y_{i,n} \in R_{\geq 2}.
\end{eqnarray*}

Similarly,  for all $\alpha \in \gc 1,n-1 \dc $,  there exist $\lambda_{1,\alpha} \in  \mathbb{K}^*$ such that 
\begin{eqnarray*}
\label{eq5}
\sigma(Y_{1,\alpha})-  \lambda_{1,\alpha} Y_{1,\alpha} \in R_{\geq 2}.
\end{eqnarray*}

To conclude it just remains to prove that there exists 
$\lambda_{1,n}\in  \mathbb{K}^*$ such that 
$\sigma(Y_{1,n})-  \lambda_{1,n} Y_{1,n} \in R_{\geq 2}$. 
This follows easily from 
Lemma~\ref{graduation2} and the fact that $\sigma \in G'$.
\end{proof}

\subsection{Summary}

Recall that we conjecture that $\aut (R)$ is the semi-direct 
product of $\mathcal{H}$ and the subgroup of order two generated 
by the transposition $\tau$. We set 
$G= \mathcal{H} \rtimes \langle \tau \rangle$.  
The previous result shows that for all $\sigma \in \aut(R)$, 
there exists $g \in G$ such that 
$$g \circ \sigma (Y_{\ia}) -Y_{\ia} \in R_{\geq 2}$$
for all $(\ia ) \in \gc 1,n \dc ^2$. 

So to prove Conjecture \ref{conj} it is enough to prove that the only automorphism $\sigma$ of $R$ such that 
\begin{equation}
\label{degreeauto}
\sigma (Y_{\ia} )-Y_{\ia} \in R_{\geq 2},
\end{equation}
for all $(\ia ) \in \gc 1,n \dc ^2$, is the identity automorphism. 

Automorphisms satisfying the above property (\ref{degreeauto}) are 
closely related to derivations of $R$. Indeed, let $D(R)$ denote  
the subalgebra of $\mathrm{End}_{\mathbb{K}}(R)$ generated by the 
$\mathbb{K}$-linear derivations of $R$. Alev and Chamarie proved 
\cite[Lemme 1.4.1]{ac} that there exists a family $(d_l)_{l>0}$ of 
elements of $D(R)$ such that for any element $x \in R_{i}$ we have 
\begin{equation}
\label{hder}
\sigma(x)=x+\sum_{l>0}d_l(x)
\end{equation}
with $d_l(x)$ homogeneous of degree $l+i$. In \cite{ll2}, we computed the 
derivations of the algebra $R$. 
Interestingly, it easily follows from \cite[Theorem 2.9]{ll2} that 
$d(b_i) \in \ideal{b_i}$, for each derivation $d$ of $R$.  Hence, 
the same is true for any element of $D(R)$, and so we deduce 
the following result
from the above 
discussion. 

\begin{proposition}
\label{normalfixed}
Let $\sigma \in \aut(R)$ such that $\sigma (Y_{\ia}) -Y_{\ia} \in R_{\geq 2},$
for all $(\ia ) \in \gc 1,n \dc ^2$. Then $\sigma (b_i) =b_i$ for all $i \in
\{1, ..., 2n-1\}$
\end{proposition}

\begin{proof} 
The above discussion shows that $d_l(b_i) \in \ideal{b_i}$ for all $l>0$. 
Hence, we deduce from (\ref{hder}) that $\sigma(b_i) \in \ideal{b_i}$. 
Consequently, 
$\sigma(b_i)=\lambda_i b_i$ with $\lambda_i \in
\mathbb{K}^*$, 
by Lemma~\ref{utile}. 
On the other hand, $$\sigma(b_i)=b_i +\sum_{l>0}
d_l(b_i),$$ with $d_l(b_i)$ homogeneous of degree $l+\deg(b_i)$. 
Comparing the components with degree equal to the degree of $b_i$, 
we obtain $\lambda_i=1$, so
that $\sigma(b_i)= b_i$, as desired.
\end{proof}

\section{Automorphisms of $3 \times 3 $ quantum matrices}

In this section, $R$ denotes the algebra of $3 \times 3$ quantum matrices. We
prove our conjecture in the case when $n=3$. As explained in the previous
section, all we need to do is to prove that the only automorphism $\sigma \in
\aut(R)$ such that $$\sigma (Y_{\ia} )-Y_{\ia} \in R_{\geq 2},$$ for all $(\ia
) \in \gc 1,3 \dc ^2$, is the identity automorphism. 
Observe that for such an
automorphism, $\sigma (Y_{\ia} )=Y_{\ia}$ 
if and only if $\deg(\sigma (Y_{\ia} )) =1$. 

\begin{lemma} \label{lemma-sigma-qminors} 
Let $[I|\Lambda]$ be a $t\times t$ quantum minor and suppose that $\sigma$ is
an automorphism such that $\sigma (Y_{\ia} )-Y_{\ia} \in R_{\geq 2},$ for all
$(\ia ) \in \gc 1,3 \dc ^2$. Then $\sigma([I|\Lambda])-[I|\Lambda]\in R_{\geq
t+1}$. As a consequence, $\sigma([I|\Lambda])=[I|\Lambda]$ if and only if
$\deg(\sigma([I|\Lambda]))=t$.
\end{lemma} 

\begin{proof} Easy, by induction, with $t=1$ being given by the observation 
immediately preceding the statement of this lemma. 
\end{proof}

Let $\sigma \in \aut(R)$ be 
such that $$\sigma (Y_{\ia} )-Y_{\ia} \in R_{\geq
2},$$ for all $(\ia ) \in \gc 1,3 \dc ^2$. 

Set $d_{\ia}:= \deg(\sigma(Y_{\ia}))$, 
for all $(\ia ) \in \gc 1,3 \dc
^2$. Our aim is to prove that
$d_{\ia}=1$ for all $(\ia)$; so that $\sigma$ is then the 
identity automorphism. 
We note first that $d_{1,3}=d_{3,1}=1$ by
Proposition~\ref{normalfixed}. \\

In the following lemma, we will use several times the anti-endomorphism 
$\Gamma:\oqmn\rightarrow\oqmn$ defined on generators by $\Gamma(Y_{\ia})
=(-q)^{i-\alpha}[\widetilde{\alpha}|\,\widetilde{i}]$, see \cite[Corollary 
5.2.2]{pw}. Here, if $I \subseteq \{1,\dots ,n\}$, then 
$\widetilde{I}:= \{1,\dots ,n\} \setminus I$, and 
$\widetilde{i}:=\widetilde{\{i\}}$ for any $i \in \{1, \dots ,n\}$. 
The effect of $\Gamma$ on $2\times 2$ 
quantum minors is given by 
$\Gamma([I|\Lambda])
=(-q)^{I-\Lambda}[\widetilde{\Lambda}|\widetilde{I}]\Delta$, see 
\cite[Lemma 4.1]{klr}, where the superscript $I-\Lambda$ 
denotes the difference between the sum of the entries of $I$ 
and the sum of the entries of $\Lambda$.

\begin{lemma}
\label{auto33step1}
Let $\sigma \in \aut(R)$ be 
such that $$\sigma (Y_{\ia} )-Y_{\ia} \in R_{\geq
2},$$ for all $(\ia ) \in \gc 1,3 \dc ^2$. Then $d_{1,1}=d_{3,3}=1$.
\end{lemma}

\begin{proof}
Assume to the contrary that $d_{1,1}+d_{3,3} >2$. 

Recall from Proposition \ref{normalfixed} that 
$b_2 = \sigma(b_2)= \sigma(Y_{1,2})\sigma(Y_{2,3})
-q\sigma(Y_{1,3})\sigma(Y_{2,2})$, so that 
$$b_2 =  \sigma(Y_{1,2})\sigma(Y_{2,3})-qY_{1,3}\sigma(Y_{2,2}).$$
Hence, comparing the degrees on both sides, we obtain  
\begin{equation*}
d_{1,2}+d_{2,3}= 1+d_{2,2}.
\end{equation*}
Similarly, by using $b_4$, we obtain  
\begin{equation*}
d_{2,1}+d_{3,2}= 1+d_{2,2}.
\end{equation*}

Suppose that 
$d_{1,1}+d_{2,2}\leq d_{1,2}+d_{2,1}$ and that 
$d_{2,2}+d_{3,3}\leq d_{2,3}+d_{3,2}$. Then 
$d_{1,1}+2d_{2,2}+d_{3,3}\leq d_{1,2}+d_{2,1}+d_{2,3}+d_{3,2}
=2+2d_{2,2}$, by using the above two equations. 
It follows that $d_{1,1}=d_{3,3}=1$, a contradiction to the 
initial assumption. 

So either $d_{1,1}+d_{2,2} > d_{1,2}+d_{2,1}$ or $d_{2,2}+d_{3,3}>d_{2,3}+d_{3,2}$. By
symmetry, we can assume that $ d_{1,1}+d_{2,2} > d_{1,2}+d_{2,1}$. In this case,
we easily get that $\deg(\sigma([1,2|1,2])) =d_{1,1}+d_{2,2}$.

Applying $\Gamma$ to the equation $[1,3|1,3]=Y_{1,1}Y_{3,3}-qY_{1,3}Y_{3,1}$ 
gives 
$Y_{2,2}[1,2,3|1,2,3] =[1,2|1,2][2,3|2,3]-q[2,3|1,2][1,2|2,3]$. Thus 
$$\sigma(Y_{2,2}) \Delta
=\sigma([1,2|1,2])\sigma([2,3|2,3]) -q[2,3|1,2][1,2|2,3].$$ 
Comparing degrees, we
obtain: $$d_{2,2}+3=d_{1,1}+d_{2,2}+e,$$ where $e:=\deg(\sigma([2,3|2,3]))\geq 2$.
This forces $d_{1,1}=1$ and $e=2$, so that $\sigma(Y_{1,1})=Y_{1,1}$ and
$\sigma([2,3|2,3])=[2,3|2,3]$.

Applying $\sigma$ to the quantum Laplace expansion 
$\Delta=Y_{1,1}[2,3|2,3]-qY_{1,2}[2,3|1,3]+q^2Y_{1,3}[2,3|1,2]$, we obtain: 
$$
\Delta=Y_{1,1}[2,3|2,3]-q\sigma(Y_{1,2})\sigma([2,3|1,3])+q^2Y_{1,3}[2,3|1,2].
$$
Hence, $\sigma(Y_{1,2})\sigma([2,3|1,3])=Y_{1,2}[2,3|1,3]$. 
Thus,  
$\sigma(Y_{1,2})=Y_{1,2}$ and $\sigma([2,3|1,3])=[2,3|1,3]$. 
Similarly, we obtain
$\sigma(Y_{2,1})=Y_{2,1}$ and $\sigma([1,3|2,3])=[1,3|2,3]$. 

So $\sigma$ acts as identity on the following elements of $R$: $Y_{3,1}$,
$Y_{2,1}$, $Y_{1,1}$, $Y_{1,2}$, $Y_{1,3}$, $[1,2|2,3]$, $[1,3|2,3]$, $[2,3|2,3]$,
$[2,3|1,3]$ and $[2,3|1,2]$. 


Applying $\Gamma$ to 
$[1,3|1,2]=Y_{1,1}Y_{3,2}-qY_{1,2}Y_{3,1}$ produces 
\begin{eqnarray*}
Y_{3,2}\Delta & = &[1,3|1,2][2,3|2,3]-q[2,3|1,2][1,3|2,3]\\
&=&
\{Y_{1,1}Y_{3,2}-qY_{1,2}Y_{3,1}\}[2,3|2,3]-q[2,3|1,2][1,3|2,3]
\end{eqnarray*}
which can be re-arranged to give 
\[
\left\{\Delta - Y_{1,1}[2,3|2,3]\right\}Y_{3,2}
=
-q\left\{Y_{1,2}Y_{3,1}[2,3|2,3]  + [2,3|1,2][1,3|2,3]\right\}.
\]
In this equation, all terms except  $Y_{3,2}$ are already known 
to be fixed by $\sigma$; so $\sigma(Y_{3,2})=Y_{3,2}$ also. 

Finally, all terms in $[2,3|1,2]=Y_{2,1}Y_{3,2}-qY_{2,2}Y_{3,1}$ 
except  $Y_{2,2}$ are now known 
to be fixed by $\sigma$; so $\sigma(Y_{2,2})=Y_{2,2}$ and  
$d_{2,2}=1$. As we have already shown that
$d_{1,1}=1$, we obtain $d_{1,1}+d_{2,2}=2=d_{1,2}+d_{2,1}$, a contradiction!
\end{proof}


\begin{proposition}
\label{allfixed}
Let $\sigma \in \aut(R)$ be 
such that 
$\sigma (Y_{\ia}) -Y_{\ia} \in R_{\geq 2},$ for all $(\ia ) 
\in \gc 1,3 \dc ^2$. 
Then $\sigma (Y_{\ia}) =Y_{\ia}$ for all $i,\alpha \in \{1, 2,3\}$
\end{proposition}
\begin{proof}
It is enough to prove that $d_{i,\alpha}=1$ 
for all $i,\alpha \in \{1, 2,3\}$. 

We already know from Proposition \ref{normalfixed} and Lemma \ref{auto33step1} that $\sigma$ leaves invariant the following quantum minors: 

$$Y_{3,1}, ~ Y_{1,1}, ~Y_{1,3}, ~Y_{3,3}, ~ [1,2|2,3], ~[1,3|1,3], ~[2,3|1,2], ~[1,2,3|1,2,3].$$ 

One can easily check that 
$$[1,2|1,3][1,3|2,3]=Y_{1,3}[1,2,3|1,2,3]+q[1,3|1,3][1,2|2,3],$$ 
by applying $\Gamma$ to the formula for $[1,2|2,3]$ and re-arranging. 
As all the minors on the right-hand side are left invariant by $\sigma$, this
implies
$$\sigma([1,2|1,3][1,3|2,3])=[1,2|1,3][1,3|2,3].$$
As usual, it follows that $\sigma([1,2|1,3])=[1,2|1,3]$ and
$\sigma([1,3|2,3])=[1,3|2,3]$.

Similarly, one obtains: $\sigma([1,3|1,2])=[1,3|1,2]$ and
$\sigma([2,3|1,3])=[2,3|1,3]$.


By a quantum Laplace expansion, we have:
$$[1,3|1,3]Y_{2,1}=q[2,3|1,3]Y_{1,1}+q^{-1}[1,2|1,3]Y_{3,1}.$$ 
As all of the minors on the
right-hand side are left invariant by $\sigma$, this implies
$$\sigma(Y_{2,1}[1,3|1,3])=Y_{2,1}[1,3|1,3].$$ 
As usual, this implies that 
$\sigma(Y_{2,1})=Y_{2,1}$ (and $\sigma([1,3|1,3])=[1,3|1,3]$).

Similarly, one can prove that $\sigma(Y_{1,2})=Y_{1,2}$, $\sigma(Y_{3,2})=Y_{3,2}$
and $\sigma(Y_{2,3})=Y_{2,3}$. 

It just remains to prove that $\sigma(Y_{2,2})=Y_{2,2}$. This easily follows
from the facts that $[1,2|2,3]=Y_{1,2}Y_{2,3}-qY_{2,2}Y_{1,3}$ and that $\sigma$
leaves invariant all these quantum minors except maybe $Y_{2,2}$.
\end{proof}

From this proposition and Proposition~\ref{reduction}, 
we deduce our main theorem:

\begin{theorem} The automorphism group of the algebra of $3\times 3$ quantum 
matrices is the semidirect product of the torus automorphisms and the 
cyclic group of order $2$ given by the transpose automorphism. 
\end{theorem} 

After this article was completed, Conjecture~\ref{conj} was proved in \cite{yak}.


\ \\ 

\newpage \noindent
St\'ephane Launois\\ 
School of Mathematics, Statistics \& Actuarial Science,\\ 
University of Kent, \\
Canterbury, Kent CT2 7NF, United Kingdom\\
E-mail: {\tt S.Launois@kent.ac.uk}
 \\[10pt]
Tom Lenagan\\
Maxwell Institute for Mathematical Sciences,\\
School of Mathematics, University of Edinburgh,\\
James Clerk Maxwell Building, King's Buildings, Mayfield Road,\\
Edinburgh EH9 3JZ, Scotland, UK\\
E-mail: {\tt tom@maths.ed.ac.uk}

\end{document}